\pdfoutput=1

\documentclass{amsart}

\usepackage[bookmarks=true,%
    colorlinks=true,%
    linkcolor=blue,%
    citecolor=blue,%
    filecolor=blue,%
    menucolor=blue,%
    urlcolor=blue,%
    breaklinks=true]{hyperref}

\usepackage{bm}
\usepackage{amsthm,amsmath,amssymb,tikz,tikz-cd}
\usepackage[mathscr]{euscript}
\usepackage{graphicx}

\newcommand{\Id}{\mathrm{Id}}
\newcommand{\C}{\mathbb{C}}
\renewcommand{\H}{\mathbb{H}}
\newcommand{\Q}{\mathbb{Q}}
\newcommand{\Z}{\mathbb{Z}}
\newcommand{\R}{\mathbb{R}}

\newcommand{\SL}{\mathrm{SL}}
\newcommand{\GL}{\mathrm{GL}}

\renewcommand{\qed}{ $\sqcup\!\!\!\!\sqcap$}
\newcommand{\G}{\Gamma}

\newtheorem{theorem}{Theorem}[section]
\newtheorem{lemma}[theorem]{Lemma}
\newtheorem{corollary}[theorem]{Corollary}

\newtheorem{remark}[theorem]{Remark}

\newtheorem{proposition}[theorem]{Proposition}
\newtheorem{question}[theorem]{Question}

\title{Complex hyperbolic 2-orbifolds with isolated singularities}
\author{A. W. Reid \& M. Stover} \thanks{The second author was supported by NSF grant DMS-2203555}

\def\qed{ $\sqcup\!\!\!\!\sqcap$}

\def\PU{\mbox{\rm{PU}}}
\def\SU{\mbox{\rm{SU}}}
\def\U{\mbox{\rm{U}}}

\begin{document}

\begin{abstract} For each prime $p$, this paper constructs compact complex hyperbolic $2$-manifolds with an isometric action of $\Z / p \Z$ that is not free and has only isolated fixed points. The case $p = 2$ is special, and finding general examples for $p=2$ is related to whether or not complex hyperbolic lattices are conjugacy separable on torsion.
\end{abstract}

\maketitle

%
%
%
%
\section{Introduction}
\label{intro}

The purpose of this note is to prove the following result.

\begin{theorem}
\label{there_are_isolated}
For all primes $p$, there is a complex hyperbolic surface $X$ with an action of $\Z/p\Z$ that is not free and has only isolated fixed points.
\end{theorem}

The existence of such actions was asked of us by J. Lott in the case of $p=2$ in connection with his work on Ricci flow on K\"ahler surfaces (see \cite{Lott}).  The proof of Theorem \ref{there_are_isolated} is relatively straightforward for odd primes, but one does find an intriguing difference between the two constructions of arithmetic lattices, namely those of \emph{simplest} and \emph{division algebra} types. We refer the reader to \S \ref{ssec:DivTor} for restrictions on torsion in arithmetic lattices in $\SU(2,1)$ of division algebra type implying that those lattices can only be used to prove Theorem \ref{there_are_isolated} when $p \equiv 1 \pmod{3}$. The full proof of Theorem \ref{there_are_isolated} is given in \S \ref{proof_main}.

The case $p=2$ is handled through an explicit example whose existence is somewhat fortuitous. This case is more delicate due to the fact that $2$-torsion in lattices with isolated fixed points are not obviously differentiated from complex reflections of order $2$ in finite quotients. See \S \ref{final} for more on this and connections to open questions like conjugacy separability of complex hyperbolic lattices.

\medskip

We close the introduction with a remark and some terminology. If $M$ is a compact complex hyperbolic surface, $F$ is a finite group of isometries acting 
with isolated fixed points, and $X = M / F$, then $X$ is a normal complex surface with cyclic quotient singularities such that the projection $M \to X$ is unbranched (i.e., unramified in codimension one). This equips $X$ with a metric of constant holomorphic sectional curvature $-1$ whose singularities are precisely at the singularities of $X$. It seems reasonable to call such an $X$ a \emph{singular complex hyperbolic surface}, and quotients of fake projective planes give an interesting class of examples (e.g., see \cite{Keum}). It would be interesting to understand these surfaces better, especially those of small volume. In \S \ref{sec:Ex} we provide several concrete examples.\\[\baselineskip]
\noindent{\bf Acknowledgements:}~{\em We are grateful to John Lott for asking the question which led us to prove Theorem \ref{there_are_isolated}, and to the Instituto de Matem\'atica Pura e Aplicada, 
Rio de Janeiro, 
for its hospitality during the conference ``Hyperbolic manifolds, their submanifolds and fundamental groups" where the ideas for this note were fleshed out. The second author is grateful to the Korean Institute for Advanced Study for its support and hospitality during the final preparation of this paper. We are also grateful to a referee for several very helpful comments.} 

%
%
%
%
\section{Complex hyperbolic preliminaries}
\label{prelim}

To make this note relatively self-contained we recall some basic facts about complex hyperbolic lattices and their quotient orbifolds, particularly those which are arithmetic. For further details on complex hyperbolic geometry, see \cite{Goldman}. Throughout, $\H^n_\C$ will denote complex hyperbolic $n$-space, and by a complex hyperbolic lattice we mean a discrete subgroup $\G<\SU(n,1)$ such that $\H^n_\C/\G$ has finite volume. When $\Gamma$ is torsion-free the quotient is a manifold, and otherwise it is an orbifold. When $n=2$ and $\Gamma$ is torsion-free we call $X=\H^2_\C/\G$ a {\em complex hyperbolic surface}.

\subsection{Complex hyperbolic 2-orbifolds}
\label{orbifolds}

This paper will require some basic facts about torsion in complex hyperbolic lattices. An element of $\SU(2,1)$ is \emph{elliptic} if it has a fixed point in $\H^2_\C$. The geometry of an elliptic element's action on $\H^2_\C$ is reflected in its eigenvalues and how the eigenspaces behave under the hermitian form used to define $\SU(2,1)$:
\begin{enumerate}
	\item Suppose $g$ is elliptic with only two distinct eigenvalues.
	\begin{enumerate}
		\item If the $2$-dimensional eigenspace is signature $(1,1)$, then $g$ is a \emph{complex reflection} acting by the identity on a totally geodesic $\H^1_\C$ in $\H^2_\C$ and by an element of $\U(1)$ on its normal bundle.
	\item If the $2$-dimensional eigenspace is positive definite, then $g$ is a \emph{reflection through a point}. In particular, it has a unique fixed point in $\H^2_\C$ and the action of $g$ on its tangent space, which is canonically identified with the rank $2$ eigenspace, is by a nontrivial scalar multiple of the identity.
	\end{enumerate}
	\item If $g$ is elliptic with three distinct eigenvalues, then it is \emph{regular elliptic}. Here, $g$ has a unique fixed point in $\H^2_\C$ where the induced action on the tangent space to that point is by a nonscalar cyclic group.
\end{enumerate}
Note that an element with only one eigenvalue is scalar, hence it acts trivially on $\H^2_\C$. Of particular importance in this paper is the observation that an elliptic element of $\SU(2,1)$ fails to have an isolated fixed point on $\H_\C^2$ if and only if it is a complex reflection.

\begin{remark}
There are no regular elliptic elements in $\SU(2,1)$ with order two action on $\H^2_\C$, simply since one cannot select three distinct numbers from $\{\pm 1\}$. This is why the prime $2$ is exceptional for the proof of Theorem \ref{there_are_isolated}.
\end{remark}

Suppose $\G < \SU(2,1)$ is a lattice and $\Lambda < \G$ is a torsion-free, normal subgroup of finite index. Then the finite group $F = \G / \Lambda$ admits an isometric action on the manifold $\H_\C^2 / \Lambda$. Whether or not $F$ has isolated fixed points is completely characterized by the torsion elements of $\G$.

\begin{lemma}\label{fixed-pts-tor}
Suppose $X = \H_\C^2 / \Lambda$ is a finite-volume complex hyperbolic $2$-manifold and $F$ is a finite group of isometries of $X$. Writing $X / F$ as $\H_\C^2 / \G$, then $F$ has isolated fixed points if and only if $\G$ contains no complex reflections.
\end{lemma}
\begin{proof}
If $F$ acts freely, then any torsion in $\G$ is central, so there is nothing to prove. If $f \in F$ has a fixed point $y \in X$, then the lift $\widetilde{f}$ of $f$ to $\H_\C^2$ defines a torsion element of $\G$ whose action on a lift of $y$ models the action of $f$. In particular, $y$ is an isolated fixed point if and only if $\widetilde{f}$ is not a complex reflection. The lemma follows.
\end{proof}

Essential to this paper for the odd prime case of Theorem \ref{there_are_isolated} is that complex reflections and regular elliptic elements are distinguished by congruence quotients.

\begin{lemma}\label{reduce-regular}
Suppose that $E$ is a number field with ring of integers $R_E$ and $\G < \SU(2,1)$ is a lattice contained in $\SL_3(R_E)$. Assume that $\G$ contains elements $\sigma, \tau$ such that $\sigma$ is a complex reflection and $\tau$ is regular elliptic. Then for all but finitely many prime ideals $\mathcal{Q}$ of $R_E$, the images of $\sigma$ and $\tau$ in $\SL_3(R_E / \mathcal{Q})$ under reduction modulo $\mathcal{Q}$ are not conjugate in $\SL_3(R_E / \mathcal{Q})$ and, moreover, the image of $\tau$ is nontrivial.
\end{lemma}
\begin{proof}
Since $\sigma$ has a repeated eigenvalue over $\C$ and $\tau$ has distinct eigenvalues, their characteristic polynomials $\chi_\sigma(t)$ and $\chi_\tau(t)$ are different elements of $R_E[t]$. Thus there are only finitely many prime ideals $\mathcal{Q}$ of $R_E$ so that $\chi_\sigma(t)$ and $\chi_\tau(t)$ are congruent modulo $\mathcal{Q}$, i.e., define the same element of $(R_E / \mathcal{Q})[t]$. If the reductions $\widehat{\sigma}$ and $\widehat{\tau}$ of $\sigma$ and $\tau$ are conjugate in $\SL_3(R_E / \mathcal{Q})$, then they have the same characteristic polynomial over the finite field $R_E / \mathcal{Q}$, hence there are only finitely many primes $\mathcal{Q}$ for which $\widehat{\sigma}$ and $\widehat{\tau}$ are conjugate in $\SL_3(R_E / \mathcal{Q})$. Applying the same argument to $(t - 1)^3$ ensures that the image of $\tau$ is nontrivial.
\end{proof}

\subsection{Arithmetic complex hyperbolic surfaces}
\label{ss:arith_cx_hyp}

Throughout this discussion, $k$ will be a totally real number field and $E/k$ a CM-extension (i.e., a totally imaginary quadratic extension of $k$). The rings of integers of these field will be denoted $R_k$ and
$R_E$ respectively. We will also fix an embedding of $E\hookrightarrow \C$ and refer to this as the identity embedding of $E$; note that complex conjugation restricts to the identity on $k$.

We now recall the construction of arithmetic subgroups of $\SU(2,1)$, beginning with those of \emph{simplest type}. Let $V$ be a $3$-dimensional vector space over $E$ equipped with an $E$-defined Hermitian form. That is, $H\in \GL(3,E)$ satisfies $H =\overline{H}^t$. Furthermore, we also assume that $H$ has signature $(2,1)$, and that for any nonidentity embedding $\nu : k \hookrightarrow \C$ and one (hence either) extension of $\nu$ to an embedding of $E$, the matrix $H^\nu$ obtained by applying $\nu$ entry-wise has signature $(3,0)$.
We can and will assume that the entries of $H$ lie in $R_E$. Define the algebraic group $\SU(H)$ as
	\[ \SU(H) = \{g\in\SL_3(\C) : gH\overline{g}^t=H\}. \]
The subgroup $\SU(H,R_E) = \SU(H) \cap SL(3,R_E)$ is an arithmetic lattice of $\SU(H,\R)$ that, after appropriate conjugation in $\GL(3,\C)$ to take $H$ to the standard form, determines an arithmetic lattice in $\SU(2,1)$. The arithmetic lattices of simplest type are those determined by the totality of all commensurability classes of the groups described above. Finally, $\SU(H, R_E)$ is cocompact if and only if $k \neq \Q$.

\begin{remark}
\label{some_remarks}
In \cite{McR}, it is shown that if $G$ is a finite subgroup of $\U(2) \times \U(1)$, then there exists an arithmetic lattice of simplest type that contains a subgroup isomorphic to $G$.
\end{remark}

The other arithmetic lattices are those of \emph{division algebra type}. Continuing with the above notation, let $A$ be a division algebra over $E$ with center $E$ and degree three over $E$, so then $A$ is a $9$-dimensional $E$-vector space. Suppose that $A$ also admits an involution $\tau$ such that the restriction $\tau|_E$ from the natural inclusion $E \hookrightarrow A$ is complex conjugation; such an involution is called an \emph{involution of the second kind}. A Hermitian element of $A$ is an element $h\in A^*$ such that $\tau(h) = h$. Now $A\otimes_E \C \cong M(3,\C)$ and we insist that $h$ has signature $(2,1)$ and has signature $(3,0)$ after applying any embedding $\sigma : E\hookrightarrow \C$ extending a nonidentity embedding of $k$.

Define the algebraic group $\SU(h,A)$ as:
	\[ \SU(h,A) =\! \left\{x \in A^1 : \tau(x)hx = h\right\}\!, \]
where $A^1$ denotes the elements of $A$ with reduced norm $1$; note that the reduced norm becomes the determinant under any matrix embedding of $A$. Let $\mathcal{O}\subset A$ be a $R_E$-order of $A$ and define the arithmetic lattice $\SU(h,\mathcal{O}) = \SU(h,A) \cap \mathcal{O}$. This then defines an arithmetic lattice in $\SU(2,1)$. The arithmetic lattices of division algebra type are those determined by the totality of all commensurability classes of the groups described above.

\begin{remark}\label{some_remarks2}
All arithmetic lattices of division algebra type are cocompact. In fact, it is well-known to experts that taking $k=\Q$ and $E=\Q(\sqrt{-d})$ for some square-free $d > 0$ in the construction of lattices of simplest type provides the only commensurability classes that are not cocompact.
\end{remark}

The next two subsections study torsion in arithmetic lattices of simplest type and division algebra type, respectively.

\subsection{The groups $G(m,p,2)$}

This section describes a special case of the results in \cite{McR} in some detail. Fix integers $m, p \ge 2$ such that $p$ divides $m$. In what follows, $\mu_m$ will denote the group of all $m^{th}$ roots of unity. First, consider the abelian subgroup
	\[ A(m,p,2) =\!\left\{\begin{pmatrix} \xi_1 & 0 \\ 0 & \xi_2 \end{pmatrix}\ :\ \xi_1, \xi_2 \in \mu_m,\, (\xi_1 \xi_2)^{\frac{m}{p}} = 1 \right\} \]
of $\U(2)$, which is normalized by the element $w \in \U(2)$ of order two acting by coordinate permutations. Then $G(m,p,2) = A(m,p,2) \rtimes \langle w \rangle$. Note that $G(m,m,2)$ is simply the dihedral group of order $2 m$. The following records some standard facts about $G(m,p,2)$.

\begin{proposition}\label{G(m,p,2)props}
The group $G(m,p,2)$ has the following properties:
\begin{enumerate}

\item $G(m,p,2)$ preserves the standard positive definite hermitian form on $\C^2$;

\item $|G(m,p,2)| = \frac{2 m^2}{p}$;

\item complex reflections in $G(m,p,2)$ have order $2$ or dividing $\frac{m}{p}$;

\item complex reflections in $G(m,p,2)$ not of order $2$ are contained in $A(m,p,2)$;

\item if $p \ge 3$ is odd and $\gcd(\frac{m}{p}, p) = 1$, then no power of an element of $G(m,p,2)$ of order $p$ is a complex reflection;

\item if $p \ge 3$, then $G(m,p,2)$ contains an element of order $p$ with distinct nontrivial eigenvalues.

\end{enumerate}
\end{proposition}
\begin{proof}
The first statement is clear from the fact that diagonal matrices and $w$ all preserve the standard inner product on $\C^2$. See \cite[\S 2.2]{URG} for the second statement and \cite[Lem.~2.8]{URG} for the third and fourth. The fifth statement is simply that no nontrivial divisor of $p$ divides $2$ or $\frac{m}{p}$. For the last statement, the element
\begin{equation}\label{gp}
g_p = \begin{pmatrix} \zeta_p & 0 \\ 0 & \zeta_p^{-1} \end{pmatrix}
\end{equation}
is in $G(m,p,2)$ for any primitive $p^{th}$ root of unity $\zeta_p$ and has distinct eigenvalues if $p \ge 3$.
\end{proof}

In Proposition \ref{G(m,p,2)props}, \emph{complex reflection} means the classical definition: an element of $\U(2)$ with $1$ as an eigenvalue. However, this ends up having the same meaning once we let $G(m,p,2)$ act on $\H_\C^2$. Specifically, take $\U(2,1)$ with respect to the standard hermitian form $|z_1|^2 + |z_2|^2 - |z_3|^2$ on $\C^3$, which by Proposition \ref{G(m,p,2)props}(1) gives an embedding
\begin{equation}\label{Gembed}
\begin{pmatrix} \xi_1 & 0 \\ 0 & \xi_2 \end{pmatrix} \mapsto \begin{pmatrix} \xi_1 & 0 & 0 \\ 0 & \xi_2 & 0 \\ 0 & 0 & 1 \end{pmatrix}
\end{equation}
of $A(m,p,2)$ in $\U(2,1)$. Note that $G(m,p,2)$ then embeds in $\PU(2,1)$ under projection, so it acts faithfully on $\H^2_\C$.

\medskip

\noindent
\textbf{Convention:} For the remainder of this paper, $G(m,p,2)$ is considered as embedded in $\U(2,1)$ by Equation \eqref{Gembed}.

\medskip

Fortunately, the two notions of complex reflection coincide.

\begin{lemma}\label{still-a-reflection}
An element $g \in A(m,p,2)$ is a complex reflection in $G(m,p,2)$ if and only if it acts on $\H^2_\C$ as a complex reflection, and $g$ acts as a regular elliptic if and only if $\xi_1 \neq \xi_2$ and neither equals one.
\end{lemma}
\begin{proof}
Inspection of Equation \eqref{Gembed} shows that $g$ is a complex reflection as an element of $\U(2)$ if and only if it has $1$ as a repeated eigenvalue as an element of $\U(2,1)$, and $g$ then has a repeated eigenvalue for which the $2$-dimensional eigenspace has signature $(1,1)$. Equivalently, $\det(g)^{- \frac{1}{3}} g \in \SU(2,1)$, which has the same action on $\H_\C^2$ as $g$, has a repeated eigenvalue with eigenspace of signature $(1,1)$, which means precisely that $g$ acts on $\H_\C^2$ as a complex reflection. A similar argument shows that $\xi_1 \neq \xi_2$ with neither equal to one if and only if $g$ acts as a regular elliptic.
\end{proof}

The following last fact will be used in the sequel.

\begin{corollary}\label{RegularElliptic}
Let $p$ be an odd prime, $m$ be divisible by $p$, and suppose that $\G < \SU(2,1)$ contains
\begin{equation}\label{SpecialG}
S(m, p, 2) = G(m, p, 2) \cap \SU(2,1).
\end{equation}
Then $\G$ contains a regular elliptic element of order $p$.
\end{corollary}
\begin{proof}
The element $g_p$ from Equation \eqref{gp} from the proof of Proposition \ref{G(m,p,2)props} is in $S(m,p,2)$, and it has distinct eigenvalues since $p$ is odd. It acts on $\H_\C^2$ as a regular elliptic element by Lemma \ref{still-a-reflection}.
\end{proof}

\subsection{Torsion in lattices of division algebra type}\label{ssec:DivTor}

Torsion in an arithmetic lattice of division algebra type is quite special. The first restriction on torsion in lattices of division algebra type is as follows.

\begin{lemma}\label{lem:DivAlgElliptic}
Suppose $\Gamma < \SU(2,1)$ is an arithmetic lattice of division algebra type and $\tau \in \Gamma$ has finite order and is not central. Then $\tau$ is regular elliptic.
\end{lemma}
\begin{proof}
Let $A$ be the division algebra associated with $\Gamma$ and $E$ its center. Since $A$ can be realized in $\mathrm{M}_3(\C)$ as the $E$-span of any finite-index subgroup of $\Gamma$, and $\tau$ normalizes some finite-index subgroup of $A^1 \cap \Gamma$, $\tau$ normalizes $A$. By the Skolem--Noether theorem, there is some scalar matrix $\xi$ (i.e., an element of the kernel of the conjugation action) so that $\sigma = \xi \tau$ is in $A^*$. Note that $\sigma$ has distinct eigenvalues if and only if $\tau$ does. Then $\sigma$ has a cubic characteristic polynomial $f(t) \in E[t]$. If $\tau$ does not have distinct eigenvalues, then $f(t)$ has a linear factor $(t - \lambda)$ for some $\lambda \in E$. It now follows from the Cayley--Hamilton Theorem that $\tau - \lambda \Id$ is a zero-divisor in $A$, contradicting the fact that $A$ is a division algebra. This completes the proof.
\end{proof}
The following generalizes an observation due to Prasad and Yeung \cite[Lem.\ 9.2]{PY}.

\begin{lemma}\label{lem:DivAlgp=1(3)}
Suppose that $\Gamma < \SU(2,1)$ is of division algebra type and $\tau \in \Gamma$ acts on $\H_\C^2$ as an elliptic element of order $n$. Then $\varphi(n)$ or $\varphi(3 n)$ is divisible by $3$, where $\varphi$ is Euler's totient function. In particular, if $n = p > 3$ is prime then $p \equiv 1 \pmod{3}$.
\end{lemma}
\begin{proof}
It suffices to consider $\Gamma$ a maximal lattice in $\SU(2,1)$. It is known that $\Gamma \cap A^1$ is a normal subgroup of $\Gamma$ with index a power of $3$ (see the Remark after Eq.\ (0) in \cite{PY}). Since $3 \mid \varphi(9)$, it suffices to prove the lemma for $\tau \in A^1$, which implies that $\tau^n$ is central in $A$. Moreover, the subalgebra $E(\tau)$ is a commutative subalgebra, hence it is a finite extension of $E$. Note that $\tau \notin E$, since otherwise it is central in $A$. Therefore $E(\tau) / E$ is an extension of degree three.

Since $\tau^n$ has reduced norm one, and the reduced norm on the center is simply $z \mapsto z^3$, it follows that $\tau^n$ is a $3^{rd}$ root of unity. In particular, $\tau$ is either an $n^{th}$ or $(3 n)^{th}$ root of unity, hence its minimal polynomial over $\Q$ has degree $\varphi(k n)$ for $k \in \{1,3\}$. Since $\Q(\tau)$ is then Galois over $\Q$, it follows that $\overline{E}(\tau)$ is Galois over the Galois closure $\overline{E}$ of $E$ over $\Q$, with degree $1$ or $3$.

First, suppose that $\overline{E}(\tau) / \overline{E}$ is degree $3$. The diagram
	
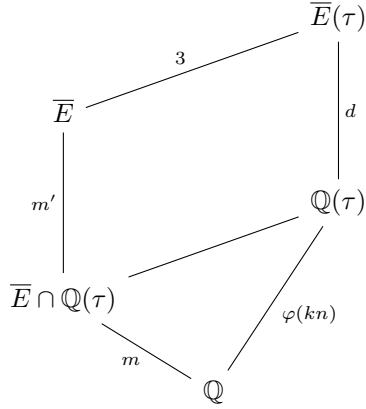
\[ \begin{tikzcd} & & \overline{E}(\tau) \\ \overline{E} \arrow[urr, -, "3"] & & \\ & & \Q(\tau) \arrow[uu, -, "d" right] \\ \overline{E} \cap \Q(\tau) \arrow[uu,-, "m^\prime"] \arrow[urr,-] & & \\ & \Q \arrow[ul, -, "m"] \arrow[uur, -, "\varphi(kn)" below right] & \end{tikzcd} \]
then consists only of Galois extensions. It is then a standard exercise in Galois theory to prove that $[\Q(\tau) : \overline{E} \cap \Q(\tau)] = 3$ and that this degree divides $\varphi(k n)$.

Now suppose that $E(\tau) \subseteq \overline{E}$. Let $G$ denote the Galois group of $E$ over $\Q$, $G_E$ be the stabilizer of $E$ in $G$, and $G_\tau$ be the stabilizer of the subfield $\Q(\tau)$ of $\overline{E}$. Since $E(\tau) / E$ is degree three, there is a diagram
	
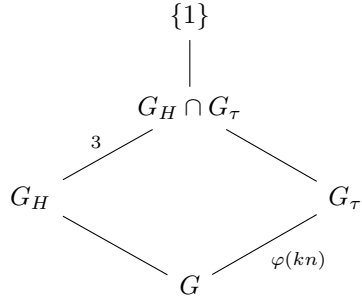
\[ \begin{tikzcd} & \{1\} \arrow[d, -] & \\ & G_H \cap G_\tau \arrow[dr, -] \arrow[dl, -, "3" above left] & \\ G_H \arrow[dr, -] & & G_\tau \arrow[dl, -, "\varphi(kn)"] \\ & G & \end{tikzcd} \]
of subgroups of $G$. The image of $G_H$ in $G / G_\tau \cong \mathrm{Gal}(\Q(\tau) / \Q)$ is then $G_H / (G_H \cap G_\tau) \cong \Z / 3 \Z$. Thus $[\Q(\tau) : \Q]$ is again divisible by three, as desired. Finally, note that if $n = p > 3$ is a prime not equal to $3$, then $\varphi(3 p) = 2 (p-1)$, hence $p \equiv 1 \pmod{3}$ for either value of $k$.
\end{proof}

In particular, one can construct many commensurability classes of lattices of division algebra type so that $
\SU(h, \mathcal{O})$ is torsion-free. More germane to the purposes of this paper, Lemma \ref{lem:DivAlgp=1(3)} implies that it is impossible to prove Theorem \ref{groups_with_torsion} for all primes using lattices of division algebra type. However, we can prove the following special case.

\begin{proposition}\label{prop:SpecialCase}
Suppose that $p$ is an odd prime and $p \equiv 1 \pmod{3}$. Then there are infinitely many distinct isomorphism classes of torsion-free lattices $\Gamma < \SU(2,1)$ so that $\H_\C^2 / \Gamma$ 
admits an isometric action by $\Z / p \Z$ that is not free and has only isolated fixed points.
\end{proposition}
\begin{proof}[Sketch of the proof]
Let $\zeta$ be a primitive $p^{th}$ root of unity. If $\varphi(p) = p-1$ is divisible by $3$, then $L = \Q(\zeta)$ contains a subfield $E$ so that $L / E$ is Galois of degree $3$. Then
	\[ k = E \cap \Q(\zeta + \zeta^{-1}) \]
is the maximal totally real subfield of $E$, and $E / k$ is quadratic. Thus $E / k$ is a CM-extension, since $E$ is Galois over $\Q$ and hence totally complex.

Let $A$ be a cyclic division algebra with center $E$ and involution of second kind containing $L$ as a maximal subfield. One can then construct an order $\mathcal{O}$ containing the ring of integers $R_L$ of $L$. Choosing an appropriate hermitian element of $A$ produces a lattice $\Gamma < \SU(2,1)$ so that $\zeta$ naturally lives in $\Gamma$. Thus $\Gamma$ contains an element of order $p$. Fix any normal, torsion-free subgroup $\Lambda < \Gamma$ of finite index. By Lemma \ref{lem:DivAlgElliptic}, all torsion elements of $\Lambda$ are regular elliptics, so the action of $\Gamma / \Lambda$ on $\H_\C^2 / \Lambda$ has only isolated fixed points by Lemma \ref{fixed-pts-tor}. There is an element of $\Gamma / \Lambda$ with order $p$ that has nontrivial fixed point set by construction, so this proves the proposition. 
\end{proof}

%
%
%
%
\section{Proof of Theorem \ref{there_are_isolated}}
\label{proof_main}

Theorem \ref{there_are_isolated} will be a consequence of our next stronger result.

\begin{theorem}
\label{groups_with_torsion}
For any prime $p \in \Z$, there exist infinitely many nonisomorphic cocompact torsion-free lattices $\Lambda_p^i < \SU(2,1)$, such that the quotient manifold $X_i = \H^2_\C/\Lambda_p^i$ admits an isometric action by $\Z / p \Z$ that is not free and has only isolated fixed points.\end{theorem}

The proof will separate the cases of $p$ an odd prime, and $p=2$. However, in both cases we will make use of the following lemma.

\begin{lemma}
\label{using_congruence}
Using the notation of \S \ref{ss:arith_cx_hyp}, suppose that $\Gamma$ is a subgroup of $\SU(H,R_E)$ that contains a regular elliptic element of order $p$. Then there are infinitely many nonisomorphic subgroups $\Lambda_i < \G$ of finite index so that the quotient orbifold $\H^2_\C/\Lambda_i$ admits an isometric action by $\Z / p \Z$ that is not free and has only isolated fixed points.
\end{lemma} 

\begin{proof} Fix an elliptic element $\tau \in \Gamma$ of order $p$ as in the hypothesis. By Lemma \ref{reduce-regular}, there are infinitely many prime ideals $\mathcal{Q}_i \subset R_E$ so that the image of $\tau$ is nontrivial modulo $\mathcal{Q}_i$ and no complex reflection in $\G$ has reduction that is $\SL_3(R_E / \mathcal{Q}_i)$-conjugate to the reduction of $\tau$. Let
\[
\rho_i : \G \to \SL_3(R_E / \mathcal{Q}_i)
\]
be the reduction homomorphism, $\Lambda_i$ be its kernel, and $\widehat{\tau}_i = \rho_i(\tau)$, which is an element of order $p$ since $\tau$ has prime order $p$ in $\G$.

Note that $\Lambda_i$ is torsion-free for all but finitely many $\mathcal{Q}_i$ by a classical result of Minkowski, so we can assume that $X_i = \H^2_\C / \Lambda_i$ is a compact complex hyperbolic manifold for all $i$. Moreover, since Strong Approximation holds for 
the ambient algebraic group $\SU(H)$, the reduction homomorphism is onto for all but finitely primes $\mathcal{Q}_i$ \cite{Wei}. Since the $R_E$ only contains finitely many prime ideals of bounded norm, the reductions modulo $\mathcal{Q}_i$ then have order going to infinity, i.e., the $X_i$ have volume going to infinity and hence are distinct.

Then $\widehat{\tau}_i$ defines an isometry of $X_i$ with order $p$. Since $\Lambda_i$ is torsion-free, no complex reflection in $\G$ is in $\Lambda_i$. It was also already noted that no complex reflection can have reduction conjugate into
	\[ F_i = \langle \widehat{\tau}_i \rangle \cong \Z / p \Z \]
inside $\G / \Lambda_i$. Thus $\G_i = \rho_i^{-1}(F_i) < \G$ contains no complex reflections. Then $F_i$ acts on $X_i$ with isolated fixed points by Lemma \ref{fixed-pts-tor}, so the proof of the lemma is complete.
\end{proof}

\subsection{Odd primes}
\label{odd}
For the case of $p$ an odd prime, we will construct a lattice $\Gamma_p < \SU(2,1)$ of simplest type containing a regular elliptic element of order $p$. Consider a $p^{th}$ root of unity $\zeta_p$, set $E = \Q(\zeta_p)$, let $k$ be its maximal totally real subfield, and consider the identity embedding of $k$ sending a generator $\zeta_p + \zeta_p^{-1}$ to $2 \cos(2 \pi / p)$.

Then $k$ has degree $d = \frac{p-1}{2}$ over $\Q$, and it is a classical fact that the various real embeddings of $k$ embed $R_k$ as a lattice in $\R^d$. In particular, there is an element $\beta \in R_k$ so that $\beta < 0$ under the identity embedding of $k$ and $\nu(\beta) > 0$ for all other embeddings $\nu$ of $k$. Thus the hermitian form $H_\beta$ given by $|z_1|^2 + |z_2|^2 + \beta |z_3|^2$ defines an arithmetic lattice of simplest type.

\begin{lemma}\label{hasgp}
The lattice $\SU(H_p, R_E)$ defined in this section contains a regular elliptic of order $p$.
\end{lemma}
\begin{proof}
Under the embedding from Equation \eqref{Gembed}, the group $G(p,p,2)$ also preserves the hermitian form $H_\beta$, and thus $S(p,p,2)$ defines a subgroup of $\SU(H_\beta, R_E)$ containing a regular elliptic of order $p$ by Corollary \ref{RegularElliptic}.
\end{proof}

\begin{proof}[Proof of Theorem \ref{groups_with_torsion} for $p$ odd]
If $p > 3$, then set $\G = \SU(H_\beta, R_E)$. This contains a regular elliptic element of order $p$ by Lemma \ref{hasgp}, and is cocompact since $k \neq \Q$. By Lemma \ref{using_congruence}, this proves Theorem \ref{groups_with_torsion}. When $p = 3$, the lattice is not cocompact, so an adjustment is necessary. Replacing $\zeta_3$ with $\zeta_9$, one still finds $S(3,3,2)$ inside $\SU(H_\beta, R_E)$, which is now cocompact since $k \neq \Q$, so the theorem again follows.
\end{proof}

\subsection{p=2}
\label{two}
To handle this case, we will make use of an example from \cite{DPP1} (see Theorem 1.2 therein with $p=3$). This example is discussed in \cite{DPP2} in more detail, where it is the group denoted by $\mathcal{S}(\overline{\sigma}_4,3)$. Henceforth we will denote this group by $\Gamma$, and note that $\G$ is cocompact. A presentation for $\G$ is:
\[ \G =\! \left\langle R_1, R_2, R_3, J~|~R_1^3, (R_1J)^7, J^3, JR_1J^{-1}=R_2, JR_2J^{-1}=R_3, R_1R_2R_1R_2=R_2R_1R_2R_1 \right\rangle\!. \]
%
%
%
See \cite[Tb.\ 5.1]{DPP2} and note that, as stated after Table 5.1, the last two relations there can be omitted. Although not visible in the presentation given above, the element $R_1R_2$ has order $12$ (again, see \cite[Tb.\ 5.1]{DPP2}), and so in particular there is an element in $\G$ of order $2$.

As shown in \cite[Thm.~1.1]{DPP1}, $\Gamma$ is arithmetic of simplest type with:
\begin{align*}
k&=\Q(\sqrt{21}) & E&=\Q(\sqrt{-3},\sqrt{-7})
\end{align*}
The fact that $\Gamma$ is simplest type can be deduced from \cite[\S 6]{DPP1}, and also follows from \cite[Thm.\ 1.4]{Sto} using the existence of a complex reflection in $\Gamma$. As discussed in \cite[\S 6.1]{DPP1}, $\Gamma$ can be conjugated to preserve a Hermitian form $H$ whose entries are algebraic integers in the field $E$ such that matrices representing $R_1$ and $J$ also have entries in $E$. From the descriptions of $R_1$ and $J$ given in \cite[p.\ 7]{DPP2} we see these matrices have determinant 1, and so from the presentation for $\Gamma$ we deduce that, since $R_2$ and $R_3$ are conjugates of $R_1$, they also have determinant $1$. Thus we can assume that $\Gamma < \SU(H,R_E)$.

Then \cite[\S 4.3.2 \& Appx.]{DPP2} classifies the finite subgroups of $\G$, which are the complex reflection groups denoted $G_4$ and $G_5$ in the standard Shephard--Todd notation. Consulting \cite[Appx.\ D]{URG}, all complex reflections in these groups have order $3$. Thus any element of order $2$ in $\Gamma$ is a reflection through a point. Applying the exact same argument as in Lemma \ref{using_congruence} now produces manifolds $X_i$ with a $\Z / 2 \Z$ action having a nonempty set of isolated fixed points. This completes the case $p = 2$, and thus completes the proof of Theorem \ref{groups_with_torsion}. \qed

%
%
%
%
\section{Examples}
\label{sec:Ex}

This section catalogues some examples for small odd primes.

\subsection{Fake projective planes}

Cartwright and Steger's 
completion of the classification of fake projective planes showed that they are all of division algebra type. They have Euler characteristic $3$, which implies by Chern--Gauss--Bonnet that they are minimal volume manifolds. Thus any automorphism has fixed points, and they are isolated by Lemma \ref{lem:DivAlgElliptic} and Lemma \ref{fixed-pts-tor}. The automorphism groups of fake projective planes are also classified, and one sees actions of $\Z/p\Z$ for $p = 3,7$ \cite{Keum}.

\subsection{The Cartwright--Steger surface}

The Cartwright--Steger surface \cite{CS} is a minimal-volume compact complex hyperbolic $2$-manifold, and it admits an action of $\Z / 3 \Z$ (e.g., see \cite[\S 5]{MSHurwitz}). One can check directly that all torsion elements are reflections through points and regular elliptics, so the action has isolated fixed points. In contrast with fake projective planes, this example is arithmetic of simplest type. Moreover, one can show that some of the $3$-torsion associated with isolated fixed points of the $\Z / 3 \Z$ action are reflections through points. Thus this example is of a slightly different kind than those produced using the methods of this paper.

\subsection{An example for $p = 5$}

For the pair $E = \Q(\zeta_5)$ and $k = \Q(\sqrt{5})$, one of the famous lattices $\G_{\Sigma \mu}$ constructed by Deligne and Mostow \cite{DM} ends up being $\PU(h, R_k)$ for the appropriate $h$. If $\mathfrak{p}_5$ is the prime dividing $5$ and $\G(\mathfrak{p}_5)$ is the congruence subgroup of $\G_{\Sigma \mu}$ of level $\mathfrak{p}_5$, then all torsion subgroups of $\G(\mathfrak{p}_5)$ are isomorphic to $(\Z / 5 \Z)^2$ generated by complex reflections. See \cite[Prop.\ 3.12]{AgolS} and the surrounding discussion.

Note that each $(\Z / 5 \Z)^2$, while generated by complex reflections, contains regular elliptic elements of order $5$. The congruence subgroup $\G(\mathfrak{p}_5^2)$ of level $\mathfrak{p}_5^2$ is torsion-free, and thus admits an isometric action of $\Z / 5 \Z$ with nonempty set of isolated fixed points by Lemma \ref{reduce-regular} and the same argument as the proof of Lemma \ref{hasgp}. One can show using known formulas for Euler characteristics of Deligne--Mostow orbifolds that $\H_\C^2 / \G(\mathfrak{p}_5^2)$ has Euler characteristic $1875$ (cf.\ Question \ref{vol} below).

%
%
%
%
\section{Final comments}
\label{final}

As mentioned in the introduction, it would be interesting to understand the singular complex hyperbolic surfaces of small volume. With Lemma \ref{fixed-pts-tor} in mind, we say that $X = \H_\C^2 / \G$ is a \emph{singular complex hyperbolic surface} if all torsion in $\G$ is reflections through points and regular elliptics. For each prime $p$, Theorem \ref{there_are_isolated} produces closed singular complex hyperbolic surfaces for which each singular point is a cyclic quotient singularity of order $p$. We expect that the methods of this paper can also produce similar examples with $p$ replaced with any natural number $n > 1$.

\begin{question}\label{vol}
Fix a prime $p$.
\begin{enumerate}
\item What is the smallest volume closed singular complex hyperbolic surface for which all singular points have order $p$?
\item What is the smallest volume closed complex hyperbolic manifold admitting an action of $\Z / p \Z$ with isolated fixed points?
\end{enumerate}
\end{question}

\begin{remark}\label{CS}
Fake projective planes and the Cartwright--Steger surface answer Question \ref{vol}(2) for $p = 3$ and $p = 7$. See \S \ref{sec:Ex}.
\end{remark}

For large $p$, the methods of this paper fail for noncompact arithmetic lattices, i.e., the \emph{Picard modular groups}. Indeed, these are contained in $\SL_3(R_d)$, where $R_d$ is the ring of integers of $\Q(\sqrt{-d})$. Independent of $d$, these groups can only have torsion of uniformly bounded order. Indeed, by Cayley--Hamilton their eigenvalues are roots of unity with bounded degree over $\Q$.

Moving to nonarithmetic lattices, there are only finitely many commensurability classes known. Moreover, there is a universal bound for the torsion orders in any commensurability class since commensurability classes of nonarithmetic lattices have a unique maximal element for inclusion \cite[p.4]{Margulis}. Thus, the following question seems extremely difficult, since it would require constructing infinitely many new commensurability classes of nonarithmetic lattices.

\begin{question}\label{GoAheadTry}
For each prime $p$, is there a noncompact finite-volume complex hyperbolic $2$-manifold $Y_p$ admitting an action of $\Z / p \Z$ with only isolated fixed points?
\end{question}

Lastly, the difficulty in the case $p = 2$ lies in the fact that we cannot distinguish between a complex reflection of order $2$ and a reflection through a point with order $2$ in congruence quotients. Recall that a group $\G$ is \emph{conjugacy separable} if for all nonconjugate pairs of elements $\sigma, \tau \in \G$, there is a finite quotient $F$ of $\G$ where the images of $\sigma$ and $\tau$ are not conjugate. Further, $\G$ is \emph{conjugacy separable on torsion} if distinct conjugacy classes of torsion elements remain nonconjugate  in some finite quotient. A positive answer to the following question would easily give many more examples satisfying the conclusion of Theorem \ref{there_are_isolated}, give a robust approach to constructing examples for $\Z / d \Z$ with $d$ not prime, not to mention the higher-dimensional analogues of the questions studied in this paper.

\begin{question}\label{conj-sep}
Are complex hyperbolic lattices conjugacy separable? Are they conjugacy separable on torsion?
\end{question}

Many fundamental groups of real hyperbolic manifolds, and in fact the broad class of virtually compact special hyperbolic groups, are known to be conjugacy separable \cite{Pavel}. However, we do not know a single complex hyperbolic lattice (other than the Fuchsian groups in $\SU(1,1)$, of course) that is conjugacy separable.

%
%
%
%

%
%
%
%
%

\medskip

\noindent Department of Mathematics,\\ Rice University,\\ Houston, TX 77005.\\ 
\noindent School of Mathematics,\\ Korea Institute for Advanced Study (KIAS),\\ Seoul, 02455, Korea.\\
\noindent Email:~alan.reid@rice.edu\\[\baselineskip]
\noindent Department of Mathematics,\\ Temple University,\\ Philadelphia, 
PA 19122.\\
\noindent Email:~mstover@temple.edu\\[\baselineskip]

\end{document}